\documentclass{amsproc}

\usepackage[pdftex,colorlinks,citecolor=black,linkcolor=black,urlcolor=black,bookmarks=false,implicit=false]{hyperref}
\def\MR#1#2{\href{http://www.ams.org/mathscinet-getitem?mr=#1}{MR#1 #2}}
\def\doi#1{\href{http://dx.doi.org/#1}{DOI #1}}

\pagespan{123}{140}
\def\publname{Contemporary Mathematics}
\def\volinfo{Volume \textbf{625}, 2014\\ \href{http://dx.doi.org/10.1090/conm/625/12495}{\texttt{http://dx.doi.org/10.1090/conm/625/12495}}}
\copyrightinfo{2014}{American Mathematical Society}

\newcommand{\sC}{\mathcal{C}}
\newcommand{\C}{\mathbb{C}}
\newcommand{\R}{\mathbb{R}}
\newcommand{\Q}{\mathbb{Q}}
\newcommand{\Z}{\mathbb{Z}}

\newcommand{\Ghat}{\widehat{G}}
\newcommand{\Hhat}{\widehat{H}}

\newcommand{\st}{\,|\,}

\newcommand{\sP}{\mathcal{P}}
\newcommand{\sQ}{\mathcal{Q}}

\theoremstyle{plain}
\newtheorem{theorem}{Theorem}
\newtheorem{corollary}[theorem]{Corollary}
\newtheorem{lemma}[theorem]{Lemma}
\newtheorem{proposition}[theorem]{Proposition}

\numberwithin{theorem}{section}
\numberwithin{equation}{section}

\theoremstyle{definition}
\newtheorem{definition}[theorem]{Definition}
\newtheorem{remark}[theorem]{Remark}

\DeclareMathOperator{\GL}{GL}
\DeclareMathOperator{\Hom}{Hom}
\DeclareMathOperator{\covol}{covol}
\DeclareMathOperator{\vol}{vol}

\begin{document}

\title[Formal duality and generalizations of Poisson summation]
{Formal duality and generalizations of the\\ Poisson summation formula}

\author{Henry Cohn}
\address{Microsoft Research New England\\
One Memorial Drive\\
Cambridge, Massachusetts 02142} \email{cohn@microsoft.com}

\author{Abhinav Kumar}
\address{Department of Mathematics\\
Massachusetts Institute of Technology\\
Cambridge, Massachusetts 02139} \email{abhinav@math.mit.edu}

\author{Christian Reiher}
\address{Mathematisches Seminar der Universit\"at Hamburg\\
Bundesstr.\ 55\\
D-20146 Hamburg, Germany}
\email{Christian.Reiher@uni-hamburg.de}

\author{Achill Sch\"urmann}
\address{Institute for Mathematics\\
University of Rostock\\
18051 Rostock, Germany}
\email{achill.schuermann@uni-rostock.de}

\thanks{The first author was supported in part by National Science Foundation grants DMS-0757765
and DMS-0952486 and by a grant from the Solomon Buchsbaum Research Fund.}

\subjclass[2010]{Primary 05B40, 11H31; Secondary 52C17}

\begin{abstract}
We study the notion of formal duality introduced by Cohn, Kumar, and
Sch\"urmann in their computational study of energy-minimizing particle
configurations in Euclidean space.  In particular, using the Poisson
summation formula we reformulate formal duality as a combinatorial
phenomenon in finite abelian groups.  We give new examples related to Gauss
sums and make some progress towards classifying formally dual
configurations. \vskip -0.75cm \ 
\end{abstract}

\maketitle

\section{Introduction}

The Poisson summation formula connects the sum of a function
over a lattice $\Lambda \subset \R^n$ with the sum of its
Fourier transform over the dual lattice $\Lambda^*$; recall
that $\Lambda^*$ is spanned by the dual basis (with respect to
the inner product) to any basis of $\Lambda$.  In fact, Poisson
summation completely characterizes the notion of duality for
lattices.  In a computational study of energy minimization for
particle configurations, Cohn, Kumar, and Sch\"urmann
\cite{CKS} found several examples of non-lattice configurations
exhibiting a similar \emph{formal duality} with respect to a
version of Poisson summation.  In this paper, we place these
examples in a broader context, produce new examples using the
theory of Gauss sums, and take the first steps towards a
classification of formally dual configurations.

Energy minimization is a natural problem in geometric optimization, which
generalizes the sphere packing problem of arranging congruent,
non-overlapping spheres as densely as possible in $\R^n$. The \emph{energy}
$E_f(\sC)$ of a configuration $\sC \subset \R^n$ with respect to a radial
potential function $f \colon \R_{>0} \to \R$ is defined to be the average
over $x \in \sC$ of the energy of $x$, which is
\[
E_f(x, \sC) = \sum_{y \in \sC\setminus\{x\}} f(|x - y|).
\]
Of course these sums might diverge or the average over $x$ might not be well
defined.  We therefore restrict $\sC$ to be a periodic configuration, i.e.,
the union of finitely many translates of a lattice in $\R^n$, and
we consider only potential functions that decrease rapidly enough at infinity
to ensure convergence.  See Section~9 of \cite{CK1} for more details.

For each potential function $f$, the energy minimization problem asks for the
configuration $\sC$ that minimizes $E_f(\sC)$ subject to fixing the point
density $\delta(\sC)$ (i.e., the number of points per unit volume).

In \cite{CKS}, the authors undertook an experimental study of energy minima
in low dimensions for Gaussian potential functions.  This is the Gaussian
core model from mathematical physics \cite{S}, and Gaussian potential
functions also play a key role in the mathematical theory of universal
optimality \cite{CK1}, because they span the cone of completely monotonic
functions of squared distance.  (If a configuration minimizes all Gaussian
potentials simultaneously, then it minimizes many others as well, such as
inverse power laws.)  For the potential function $G_c(r) := \exp(-\pi c
r^2)$, as $c \to \infty$ the potential energy for each point is dominated by
the contribution from its nearest neighbors.  In the limit, minimizing the
energy requires maximizing the distance between the nearest neighbors and
thus maximizing the density of the corresponding sphere packing. We can
therefore view energy minimization with $c$ large as a ``soft-matter''
version of sphere packing, in which small distances between particles are
allowed but heavily penalized, and we recover the hard sphere model in the
limit as $c \to \infty$.

Maximizing density is a necessary condition for optimality as $c \to \infty$,
but it is not sufficient, since two optimal sphere packings needn't have the
same energy.  For example, one may contain fewer pairs of nearest neighbors,
in which case it will have lower energy when $c$ is large.  As shown in
\cite{CK2}, the densest lattice packing in $\R^n$ fails to minimize energy
for large $c$ when $n=5$ or $n=7$.  Further results were obtained in
\cite{CKS}, which reported on numerical searches for energy minima among
periodic configurations with $1 \le n \le 9$ and a range of values of $c$.
(The results in \cite{CKS} are formulated in terms of a fixed potential
function and varying particle density, but that is equivalent to our
perspective here under rescaling to fix the density.)

The most noteworthy finding from \cite{CKS} was that in each dimension, the
energy-minimizing structures for the potential functions $G_c$ and $G_{1/c}$
seem to be \emph{formally dual} (except in certain narrow ranges of phase coexistence).
Formal duality generalizes the more
familiar notion of duality for lattices.  We will recall the definition in
Section~\ref{defns}; the key property is that if $\sP$ and $\sQ$ are formal
duals, then formal duality relates the $f$-potential energy of $\sP$ to the
$\widehat{f}$-potential energy of $\sQ$ for all potential functions $f$,
where $\widehat{f}$ is the Fourier transform of $f$. Note that $G_c$ and
$G_{1/c}$ are Fourier transforms of each other, up to scalar multiplication.

To describe the simulation results from \cite{CKS}, we will need some
notation. Let $D_n^+$ be the periodic configuration consisting of the union
of the checkerboard lattice
\[
D_n = \{ (x_1, \dots, x_n) \in \Z^n \st x_1+\dots+x_n \equiv 0 \pmod 2 \}
\]
and its translate by the all-halves vector (note that $D_n^+$ is actually a
lattice if $n$ is even), and for $\alpha >0$ let
\[
D_n^+(\alpha) = \{ (x_1, \dots, x_{n-1}, \alpha x_n) \st (x_1,\dots,x_n) \in D_n^+\}
\]
be obtained by scaling the last coordinate.

The numerical experiments in
\cite{CKS} indicate that in dimension $5$, the family of configurations
$D_5^+(\alpha)$ minimize the $G_c$-energy, with $\alpha$ some function of the
parameter $c$, except in a small interval around $c=1$ (in this interval there
is phase coexistence and the optimal configuration is probably not periodic). For instance,
as $c \to \infty$, the minima seem to approach
$D_5^+(2)$, which is the tight packing $\Lambda_5^2$ in the notation of
\cite{CS1}. Similarly, in dimension $7$ the $D_7^+(\alpha)$ family seems to
be optimal.  In three cases there are single configurations that seem
to minimize potential energy for the entire family of Gaussian potential
functions: $D_4$ in dimensional $4$, $E_8$ in dimension $8$ (consistent with the conjecture of
universal optimality from \cite{CK1}), and $D_9^+$ in dimension 9.
In dimension $6$, the energy minima are
experimentally seen to be $E_6$ and its dual for $c \to \infty$ and $c \to 0$,
respectively; around the central point $c = 1$ experiments yield the
following periodic configuration $\sP_6(\alpha)$, where $\alpha$ depends on
$c$. Let $\sP_6$ be the lattice $D_3 \oplus D_3$, along with its translates
by the three vectors
$v_1 = ( -1/2, -1/2, -1/2, 1,1,1 )$,
$v_2 =  ( 1,1,1,-1/2, -1/2, -1/2 )$,
and $v_3 = v_1 + v_2$.
Then $\sP_6(\alpha)$ is obtained from $\sP_6$ by scaling the first
three coordinates by $\alpha$ and the last three by $1/\alpha$.

Whether or not these families are the true global minima, they certainly
exhibit the phenomenon of formal duality. Namely, $D_n^+(\alpha)$ is formally
dual to an isometric copy of $D_n^+(1/\alpha)$, and $\sP_6(\alpha)$ is
formally dual to an isometric copy of itself.  See Section~VI of \cite{CKS}
for a proof for $D_n^+(\alpha)$ and a sketch of the analogous proof for
$\sP_6(\alpha)$.  Formal duality comes as a surprise, because most
configurations do not have formal duals at all.  The experimental findings
lead to a natural question: do the global minima for Gaussian potential
energy in Euclidean space always appear in families exhibiting formal
duality?  Outside of certain narrow ranges for the parameter $c$, where one
observes phase coexistence leading to aperiodic minima, all the numerical
data from \cite{CKS} is consistent with formal duality.

The structures found in \cite{CKS} have been the only known examples of
formally dual pairs other than lattices.  In this paper, we present a new
family of examples based on Gauss sums, we analyze structural properties of
formally dual configurations, and we take the first steps towards a
classification.

\section{Poisson summation formulas and duality} \label{defns}

We first recall the Poisson summation formula. Given a well-behaved function
$f \colon \R^n \to \R$ (for example, a Schwartz function, though much weaker
hypotheses will suffice), define its Fourier transform $\widehat{f} \colon \R^n \to \R$ by
\[
\widehat{f}(y) = \int_{\R^n} f(x) e^{-2\pi i \langle x, y \rangle } \, dx.
\]
Then the Poisson summation formula states that for a lattice $\Lambda \subset
\R^n$,
\[
\sum_{x \in \Lambda} f(x) = \frac{1}{\covol(\Lambda)}\sum_{y \in \Lambda^*} \widehat{f}(y),
\]
where
\[
\Lambda^* =  \{ y \in \R^n \st \langle x,y
\rangle \in \Z \textup{ for all $x \in \Lambda$}\}
\]
is the dual lattice and $\covol(\Lambda) = \vol(\R^n/\Lambda)$ is the volume
of a fundamental domain of $\Lambda$. The Poisson summation formula is a
useful identity in many areas of mathematics. For instance, it can be used to
prove analytic continuation and the functional equation for the Riemann zeta
function.

As a consequence of Poisson summation,
\[
f(0) + E_f(\Lambda) = \frac{1}{\covol(\Lambda)}\big(\widehat{f}(0)+E_{\widehat{f}}(\Lambda^*)\big)
\]
for every lattice $\Lambda$.
Here $\widehat{f}$ is an abuse of notation, in
which we treat the potential function $f \colon \R_{>0} \to \R$
as a radial function on $\R^n$.

It follows that a lattice $\Lambda$ minimizes $E_f$ among lattices with a fixed
covolume if and only if $\Lambda^*$ minimizes $E_{\widehat{f}}$.  The most
important special case is the Gaussian potential function $G_c(r) = \exp(-\pi
cr^2)$, which has $n$-dimensional Fourier transform $\widehat{G_c}(r) = c^{-n/2} \exp(-\pi
r^2/c)$. In this case Poisson summation relates $E_{G_c}(\Lambda)$ to
$E_{G_{1/c}}(\Lambda^*)$.

One could ask if there is a reasonable analogue of the Poisson summation
formula for non-lattices. The obvious generalization would be to ask for
periodic configurations $\sP$ and $\sQ$ with
\[
\sum_{x \in \sP} f(x) = \delta(\sP) \sum_{y \in \sQ} \widehat{f}(y)
\]
for all well-behaved $f$.
Here $\delta(\sP)$ is the point density of $\sP$: if $\sP$ consists of $N$
translates of a lattice $\Lambda$, then $\delta(\sP) = N/\covol(\Lambda)$.
However, the requirement above is too stringent, for it forces $\sP$ and $\sQ$ to
be lattices, by Theorem~1 in \cite{C}. Instead, we are really interested in
the differences between points in $\sP$, at least for the purposes of
potential energy, so we modify the notion of duality as follows. For a
Schwartz function $f\colon \R^n \to \R$ and a periodic configuration $\sP = \bigcup_{j=1}^N
(\Lambda + v_j)$ (where $\Lambda$ is a lattice), we let
\[
\Sigma_f(\sP) = \frac{1}{N} \sum_{j,k=1}^N \sum_{x \in
  \Lambda} f(x + v_j - v_k)
\]
be the \emph{average pair sum} of $f$ over $\sP$. It is also the average over
all points $x \in \sP$ of $\Sigma_f(x, \sP) = \sum_{y \in \sP} f(y - x)$, and
this interpretation shows that it is independent of the decomposition of
$\sP$ as $ \bigcup_{j=1}^N (\Lambda + v_j)$.  Note that when $f$ is a radial
function, this sum is related to the potential energy by $\Sigma_f(\sP) =
E_f(\sP) + f(0)$, but we do not require $f$ to be radial.

\begin{definition} \label{definition:formalduality}
  We say two periodic configurations $\sP$ and $\sQ$ in $\R^n$ are
  \emph{formally dual} to each other if $\Sigma_f(\sP) = \delta(\sP)
  \Sigma_{\widehat{f}}(\sQ)$ for every Schwartz function $f \colon \R^n \to
  \R$.
\end{definition}

For a lattice, pair sums reduce to sums over the lattice itself.
Thus, two lattices are formally dual if and only if they are actually dual.

We define formal duality only for periodic configurations,
although there may be interesting extensions to the aperiodic
case.  Note also that the formal dual of a configuration
needn't be unique.  One form of non-uniqueness is obvious: if
$\sQ$ is a formal dual of $\sP$, then so are $\sQ+t$ and
$-\sQ+t$ for all vectors $t$.  However, formal duals are not
unique even modulo these transformations.  See
Remark~\ref{remark:uniqueness} for an example.

\begin{remark}
  If $\sP$ and $\sQ$ are formally dual as above, then we
  can prove $\delta(\sP) \delta(\sQ) = 1$ by considering a steep
  Gaussian $f(x) = \exp(-\pi c |x|^2)$ and letting $c \to
  \infty$. Therefore the relation of being formally dual is symmetric.
\end{remark}

Our notion of formal duality is stronger than another version in the
literature (see, for example, the question on p.~185 of \cite{CS2}).  The
other version asks for equality only for radial functions, which is
equivalent to a statement about the average theta series.  For clarity we
call that version radial formal duality:

\begin{definition}  \label{definition:radialformalduality}
  We say two periodic configurations $\sP$ and $\sQ$ in $\R^n$ are
  \emph{radially formally dual} to each other if $\Sigma_f(\sP) = \delta(\sP)
  \Sigma_{\widehat{f}}(\sQ)$ for every \emph{radial} Schwartz function $f \colon \R^n \to
  \R$.
\end{definition}

If $\Lambda_1$ and $\Lambda_2$ are distinct lattices in $\R^n$ with the same theta
series, then $\Lambda_1$ and $\Lambda_2^*$ are radially formally dual but not
dual and hence not formally dual.  The most interesting case is when $\Lambda_1$
and $\Lambda_2$ are not isometric (for example, $D_{16}^+$ and $E_8 \oplus E_8$),
but the simplest case is when $\Lambda_2$ is a rotation of $\Lambda_1$.

The discrete analogue of radial formal duality has been
investigated in the coding theory literature, with several
striking examples such as Kerdock and Preparata codes
\cite{HKCSS}.

Radial formal duality is all one needs for studying energy under radial
potential functions, but the stronger definition arose in the examples from
\cite{CKS} and possesses a richer structure theory.  For example,
Lemma~\ref{lemma:lineartransform} below fails for radial
formal duality (let $\sP$ be $\Z^2$, let $\sQ$ be $\Z^2$ rotated by an angle
of $\pi/4$, and let $\phi$ be the diagonal matrix with entries $2$ and $1$).

We will now transform the notion of formal duality into a more combinatorial
definition about subsets of abelian groups, rather than the continuous
setting of periodic configurations and potential functions. The first step is
the following easy result, which is Lemma~2 in~\cite{CKS}.

\begin{lemma} \label{lemma:lineartransform}
Let $\sP$ and $\sQ$ be periodic configurations of $\R^n$ which are formally
dual to each other, and let $\phi \in \GL_n(\R)$ be an invertible linear
transformation of the space. Then $\phi(\sP)$ and $(\phi^t)^{-1}(\sQ)$ are
formally dual to each other.
\end{lemma}

Here $\phi^t$ is the adjoint of $\phi$ with respect to the inner product
$\langle \cdot, \cdot \rangle$ (i.e., its matrix is the transpose of that of
$\phi$).

\begin{proof}
If $f$ is any Schwartz function, then so is $g = f \circ \phi$, and
\[
\widehat{g} = \frac{1}{\det(\phi)} \widehat{f} \circ (\phi^t)^{-1}.
\]
Therefore,
\begin{align*}
\Sigma_f\big(\phi(\sP)\big) = \Sigma_{f \circ \phi}(\sP) &= \Sigma_g(\sP) \\
&= \delta(\sP) \, \Sigma_{\widehat{g}} (\sQ) \\
&= \delta(\sP)\cdot \frac{1}{\det(\phi)} \cdot \Sigma_{\widehat{f} \circ (\phi^t)^{-1}} (\sQ) \\
&= \delta\big(\phi(\sP)\big) \, \Sigma_{\widehat{f}} \big( (\phi^t)^{-1}(\sQ) \big),
\end{align*}
which shows that $\phi(\sP)$ and $(\phi^t)^{-1}(\sQ)$ are formally dual.
\end{proof}

This lemma shows that, for a periodic configuration, the property of having a
formal dual depends only on the underlying abelian group and coset structure,
rather than how the configuration is embedded into $\R^n$.  In other words,
having a formal dual is not a metric property.

For further progress in making formal duality more combinatorial, we will
need to remove the Fourier transform from the definition.  We can do so using
Poisson summation, as follows.  The statement looks complicated, but it will
be an essential tool for simplifying the duality theory.

\begin{lemma} \label{lemma:fhatbothsides}
Let $\sP = \bigcup_{j=1}^N (\Lambda + v_j)$ and $\sQ = \bigcup_{j=1}^M
(\Gamma + w_j)$ be periodic configurations with underlying lattices $\Lambda$
and $\Gamma$, respectively.  Then $\sP$ and $\sQ$ are formally dual if and
only if for all Schwartz functions $f \colon \R^n \to \R$,
\[
\sum_{y \in \Lambda^*} \widehat{f}(y) \left| \frac{1}{N} \sum_{j=1}^N e^{2 \pi i \langle v_j, y \rangle} \right|^2
=
\frac{1}{M} \sum_{j,k=1}^M \sum_{z \in \Gamma} \widehat{f}(z + w_j - w_k).
\]
\end{lemma}

\begin{proof}
Let $v \in \R^n$. By Poisson summation for the function $x \mapsto f(x + v)$,
\[
\sum_{x \in \Lambda} f(x + v) = \frac{1}{\covol(\Lambda)} \sum_{y \in \Lambda^*} e^{2\pi i \langle v , y \rangle} \widehat{f}(y).
\]
Using this, if $\sP = \bigcup_{j=1}^N (\Lambda + v_j)$, then
\begin{align*}
\Sigma_f(\sP) &= \frac{1}{N} \sum_{j,k=1}^N \sum_{x \in \Lambda} f(x + v_j - v_k) \\
&= \frac{1}{N \covol(\Lambda)} \sum_{y \in \Lambda^*} \widehat{f}(y) \left| \sum_{j=1}^N e^{2 \pi i \langle v_j, y \rangle} \right|^2 \\
&= \delta(\sP) \sum_{y \in \Lambda^*} \widehat{f}(y) \left| \frac{1}{N} \sum_{j=1}^N e^{2 \pi i \langle v_j, y \rangle} \right|^2.
\end{align*}
Formal duality holds if and only if this quantity equals
\[
\delta(\sP) \Sigma_{\widehat{f}}(\sQ) = \frac{\delta(\sP)}{M} \sum_{j,k=1}^M \sum_{z \in \Gamma} \widehat{f}(z + w_j - w_k),
\]
as desired.
\end{proof}

This lemma has powerful consequences for the cosets of $\sP$ in $\Lambda$.
Recall that for a set $A$ in an abelian group $G$, we define $A - A = \{ x -
y \st x, y \in A \}$.

\begin{corollary}
Let $\Lambda$ and $\Gamma$ be underlying lattices of formally dual
configurations $\sP$ and $\sQ$, respectively. Then $\sP - \sP \subseteq
\Gamma^*$ and $\sQ - \sQ \subseteq \Lambda^*$.
\end{corollary}

\begin{proof}
It is enough to show the latter statement, since the former follows by
symmetry.  By Lemma~\ref{lemma:fhatbothsides},
\[
\sum_{y \in \Lambda^*} \widehat{f}(y) \left| \frac{1}{N} \sum_{j=1}^N e^{2 \pi i \langle v_j, y \rangle} \right|^2
=
\frac{1}{M} \sum_{j,k=1}^M \sum_{z \in \Gamma} \widehat{f}(z + w_j - w_k)
\]
for every Schwartz function $f$.  Since $\widehat{f}$ is an arbitrary
Schwartz function, this forces the set $\{z + w_j - w_k \st 1\le j,k \le M \textup{ and } z \in \Gamma \}$,
which is exactly $\sQ - \sQ$, to be contained in $\Lambda^*$.
\end{proof}

The following corollary holds for exactly the same reason.

\begin{corollary} \label{corollary:phases}
Let $\sP = \bigcup_{j=1}^N (\Lambda + v_j)$ and $\sQ = \bigcup_{j=1}^M
(\Gamma + w_j)$ be periodic configurations, such that $\sP - \sP \subseteq
\Gamma^*$ and $\sQ - \sQ \subseteq \Lambda^*$. Then $\sP$ is formally dual to
$\sQ$ if and only if for every $y \in \Lambda^*$,
\[
\left| \frac{1}{N} \sum_{j=1}^N e^{2 \pi i \langle v_j, y \rangle} \right|^2
=
\frac{1}{M}\cdot \# \{(z,j,k) \st \textrm{$1 \le j,k \le M$, $z \in \Gamma$, and $y = z+w_j-w_k$}\},
\]
i.e., $1/M$ times the number of ways the coset $y + \Gamma$ can be written as
a difference of two of the $M$ cosets of $\Gamma$ in $\sQ$.
\end{corollary}

From now on, we will assume without loss of generality that $0 \in \sP$ (and
therefore $\Lambda \subseteq \sP$), and similarly $0 \in \sQ$.  We may do so
because formal duality is clearly translation-invariant.

Now $\sP = \sP - 0 \subseteq \sP - \sP \subseteq \Gamma^*$, so $\sP$
can be represented as a subset $S$ of size $N$ in the finite abelian group
$\Gamma^*/\Lambda$. Similarly, $\sQ$ corresponds to a subset $T$ of $M$
points in $\Lambda^*/\Gamma$. The natural pairing
\[
({\Gamma^*}/{\Lambda}) \times ({\Lambda^*}/{\Gamma}) \to S^1 \subset \C^*
\]
given by
\begin{equation} \label{eq:pairingdef}
\langle x + \Lambda, y + \Gamma \rangle  = e^{2 \pi i \langle x, y \rangle}
\end{equation}
identifies the two groups as duals.  In other words, we view
$\Lambda^*/\Gamma$ as the group $\Ghat$ of characters on $G :=
\Gamma^*/\Lambda$, with $\chi \in \Ghat$ acting on $g \in G$ via $\chi(g) =
\langle g,\chi \rangle$.  Note that in \eqref{eq:pairingdef}, $\langle \cdot,
\cdot \rangle$ denotes both the pairing between $G$ and $\Ghat$ and the
Euclidean inner product, but the type of the inputs makes the usage
unambiguous.  We will also canonically identify $G$ with the dual of $\Ghat$
and treat the pairing between them as symmetric.

Because $v_1, \dots, v_N \in \sP \subseteq \Gamma^*$, the quantity
\[
\left| \frac{1}{N} \sum_{j=1}^N e^{2 \pi i \langle v_j, y \rangle} \right|^2
\]
from Corollary~\ref{corollary:phases} only depends on $y$ modulo $\Gamma$.

We can now reformulate formal duality as follows. Let the Fourier transform
of a function $f \colon G \to \C$ be $\widehat{f} \colon \Ghat \to \C$,
defined by
\[
\widehat{f}(y) = \frac{1}{\sqrt{|G|}} \sum_{x \in G} f(x) \overline{\langle x,y \rangle}
 = \frac{1}{\sqrt{|G|}} \sum_{x \in G} f(x) y(-x).
\]

\begin{theorem} \label{theorem:main}
  With notation as above, let $\sP$ correspond to the translates of
  $\Lambda$ by elements of $S = \{v_1, \dots, v_N \} \subseteq G =
  \Gamma^*/\Lambda$, and $\sQ$ correspond to the translates of
  $\Gamma$ by $T = \{w_1, \dots, w_M \} \subseteq \Ghat =
  \Lambda^*/\Gamma$. Then $\sP$ and $\sQ$ are formally
  dual if and only if the following equivalent conditions hold.
\begin{enumerate}
\item \label{condition1} For every $y \in \Ghat$,
\[
\left|\frac{1}{N} \sum_{i=1}^N \langle v_i, y \rangle \right|^2
= \frac{1}{M}\cdot \# \{(j, k) \st 1 \leq j, k \leq M \textrm{ and } y=w_j-w_k\}.
\]
\item \label{condition2} For every function $f \colon G \to \C$,
\[
\frac{1}{N^{3/2}} \sum_{j,k = 1}^N f(v_j - v_k) =
\frac{1}{M^{3/2}} \sum_{j,k=1}^M \widehat{f}(w_j - w_k).
\]
\end{enumerate}
\end{theorem}

\begin{proof}
  The equivalence of statement~\eqref{condition1} and formal duality is a mild
  rephrasing of Corollary~\ref{corollary:phases}.  To see why \eqref{condition1} is equivalent
  to \eqref{condition2}, we first note that
\[
  f(x) = \frac{1}{\sqrt{|G|}}
  \sum_{y \in \Ghat} \widehat{f}(y) \langle x,y \rangle.
\]
We now have
\begin{align*}
\frac{1}{N^{3/2}} \sum_{j,k = 1}^N f(v_j - v_k) &= \frac{1}{N^{3/2}} \cdot \frac{1}{\sqrt{|G|}} \sum_{j,k=1}^N \sum_{y \in \Ghat}  \widehat{f}(y) \langle v_j - v_k,y \rangle \\
&= \frac{1}{N^{3/2}} \cdot \frac{1}{\sqrt{|G|}} \sum_{j,k} \sum_y  \widehat{f}(y) \langle v_j,y \rangle\overline{\langle v_k,y \rangle} \\
&= \frac{1}{N^{3/2}} \cdot \frac{1}{\sqrt{|G|}} \sum_y  \widehat{f}(y) \sum_{j,k} \langle v_j,y \rangle\overline{\langle v_k,y \rangle} \\
&= \frac{1}{N^{3/2}} \cdot \frac{1}{\sqrt{|G|}} \sum_y  \widehat{f}(y) \left| \sum_j \langle v_j,y \rangle \right|^2 \\
&=  \sqrt{\frac{N}{|G|}} \cdot \sum_y  \widehat{f}(y) \left| \frac{1}{N} \sum_j \langle v_j,y \rangle \right|^2.
\end{align*}
The last expression equals
\[
\frac{1}{M} \sqrt{\frac{N}{|G|}} \sum_{j,k=1}^M \widehat{f}(w_j - w_k)
\]
for every $f$ if and only if \eqref{condition1} holds.  Thus, we have shown
that \eqref{condition1} is equivalent to
\begin{equation} \label{eq:mutant}
\frac{1}{N^{3/2}} \sum_{j,k = 1}^N f(v_j - v_k) =
\frac{1}{M} \sqrt{\frac{N}{|G|}} \sum_{j,k=1}^M \widehat{f}(w_j - w_k).
\end{equation}
To complete the proof of equivalence, we will show that \eqref{condition1}
and \eqref{condition2} each imply $|G|=MN$ (in which case \eqref{eq:mutant}
is equivalent to \eqref{condition2}).

First, assume \eqref{condition1}.  If we sum over all $y \in \Ghat$ and apply
orthogonality of distinct characters on $\Ghat$, we find that
\[
\sum_{y \in \Ghat} \left|\frac{1}{N} \sum_{i=1}^N \langle v_i, y \rangle \right|^2
= \sum_{y \in \Ghat} \frac{1}{N^2} \sum_{i=1}^N |\langle v_i, y \rangle|^2 =
\frac{1}{N^2} N | \Ghat|.
\]
Thus, \eqref{condition1} yields
\[
\frac{1}{N^2} N | \Ghat| = \frac{1}{M} \cdot M^2,
\]
implying $|G| = |\Ghat| = MN$. Assuming \eqref{condition2}, we can apply it
with $f$ being the characteristic function of the identity in $G$ to obtain
\[
\frac{1}{N^{3/2}} \cdot N = \frac{1}{M^{3/2}} \cdot \frac{1}{\sqrt{|G|}} \cdot M^2,
\]
which again implies $|G| = MN$.
\end{proof}

\begin{definition} \label{def:main}
  We say that subsets $S$ of a finite abelian group $G$ and $T$ of
  $\Ghat$ are \emph{formally dual} if the following equivalent
  conditions hold.
\begin{enumerate}
\item \label{condition1d} For every $y \in \Ghat$,
\[
\left|\frac{1}{|S|} \sum_{v \in S} \langle v, y \rangle \right|^2
= \frac{1}{|T|}\cdot \# \{(w,w') \in T \times T \st y=w-w'\}.
\]
\item \label{condition2d} For every function $f \colon G \to \C$,
\[
\frac{1}{|S|^{3/2}} \sum_{v,v' \in S} f(v-v') =
\frac{1}{|T|^{3/2}} \sum_{w,w' \in T} \widehat{f}(w-w').
\]
\end{enumerate}
\end{definition}

Thus, Theorem~\ref{theorem:main} reduces formal duality in Euclidean space to
the setting of finite abelian groups.

\begin{remark}
  The second criterion in the definition immediately implies that the relation of formal
  duality is symmetric. However, the first criterion seems to be more
  useful for concrete calculations, and it is the one we will use in
  our examples.
\end{remark}

\section{Examples}

The simplest examples of formally dual configurations in $\R^n$ are of course
lattices and their duals. These correspond to taking the trivial abelian
group $G = \{0\}$, with $S = G$ and $T = \Ghat = \{0\}$.

\subsection{The TITO configuration}

The simplest non-trivial example of a pair of formally dual configurations is the
following. Consider the abelian group $G = \Z/4\Z$, and identify $\Ghat =
\Z/4\Z$ via the pairing $\langle x, y \rangle = e^{2 \pi i xy/4 }$. Let $S =
T = \{0,1\}$.  We check condition~\eqref{condition1d} of
Definition~\ref{def:main} as follows for each value of $y$:
\begin{align*}
y = 0\colon& \quad \left|\frac{1}{2}(1 + 1) \right|^2 = 1 = \frac{1}{2} \# \{ (0,0), (1,1) \}, \\
y = 1\colon& \quad \left|\frac{1}{2}(1 + i) \right|^2 = \frac{1}{2} = \frac{1}{2} \# \{ (1,0) \}, \\
y = 2\colon& \quad \left|\frac{1}{2}(1 -1) \right|^2 = 0 = \frac{1}{2} \# \{ \}, \\
y = 3\colon& \quad \left|\frac{1}{2}(1 - i) \right|^2 = \frac{1}{2} = \frac{1}{2} \# \{ (0,1) \}.
\end{align*}
Thus, $S$ and $T$ are formally dual to each other. We call this configuration
TITO, which stands for ``two-in two-out'':

\begin{center}
\setlength{\unitlength}{0.1cm}
\begin{picture}(95,6)(-15,-3)
\put(-11.59,0){\dots} \put(0,0){\circle*{1.5}} \put(10,0){\circle*{1.5}}
\put(20,0){\circle{1.5}} \put(30,0){\circle{1.5}} \put(40,0){\circle*{1.5}}
\put(50,0){\circle*{1.5}} \put(60,0){\circle{1.5}} \put(70,0){\circle{1.5}}
\put(78.41,0){\dots} 
\end{picture}
\end{center}

TITO yields the following formally self-dual configuration in one-dimensional
Euclidean space $\R$:
\[
\sP = \sQ = 2\Z \cup (2\Z + 1/2).
\]
All of the examples from \cite{CKS} described in the introduction are
products of copies of $\Z$ and the TITO configuration $\sP$, up to linear
transformations.  For example, it is not hard to check that for odd $n$ we
can obtain $D_n^+$ from the product $\sP \times \Z^{n-1}$.  (Recall that for
even $n$, $D_n^+$ is a lattice.)  Similarly, the putative optimum $\sP_6$ in
six dimensions can be obtained from $\sP^2 \times \Z^4$.  These product
decompositions imply formal duality, by the following lemma.

\begin{lemma}
Let $S_1 \subseteq G_1$ and $T_1 \subseteq \widehat{G_1}$ be formal duals,
and let $S_2 \subseteq G_2$ and $T_2 \subseteq \widehat{G_2}$ be formal
duals. Then $S_1 \times S_2 \subseteq G_1 \times G_2$ is formally dual to
$T_1 \times T_2 \subseteq \widehat{G_1} \times \widehat{G_2}$.
\end{lemma}

\begin{proof}
This follows directly from the second criterion in Definition~\ref{def:main}.
(Of course the first criterion also leads to a simple proof.)  Setting $G =
G_1 \times G_2$, $S = S_1 \times S_2$, and $T = T_1 \times T_2$ and
identifying $\Ghat$ with $\widehat{G_1} \times \widehat{G_2}$, we must show
that every function $f \colon G \to \C$ satisfies
\[
\frac{1}{|S|^{3/2}} \sum_{v,v' \in S} f(v-v') =
\frac{1}{|T|^{3/2}} \sum_{w,w' \in T} \widehat{f}(w-w').
\]
This identity follows immediately from taking the product of the
corresponding identities for $G_1$ and $G_2$ if there are functions $f_i
\colon G_i \to \C$ such that $f(x_1,x_2) = f_1(x_1) f_2(x_2)$ for all
$(x_1,x_2) \in G_1 \times G_2$. Such functions span all the functions on $G_1
\times G_2$, which completes the proof.
\end{proof}

\subsection{The Gauss sum configurations}

We now consider the case $G = (\Z/p\Z)^2$ and $\Ghat =
(\Z/p\Z)^2$, with $p$ an odd prime. The pairing is given by
\[
\langle (a,b), (c,d) \rangle = \zeta_p^{ac + bd},
\]
where $\zeta_p = e^{2\pi i/p}$.

\begin{theorem} \label{theorem:gauss}
For all nonzero elements $\alpha$ and $\beta$ of $\Z/p\Z$, the subsets $S =
\{(\alpha n^2, \beta n) \st n \in \Z/p\Z \}$ and $T = \{ (n,n^2) \st n \in
\Z/p\Z \}$ are formally dual to each other.
\end{theorem}

\begin{proof}
Recall that the absolute value squared of the classical Gauss sum
$\sum_{n=1}^p \zeta_p^{n^2}$ is $p$. It follows by completing the square that
\[
\left|\sum_{n=1}^p\zeta_p^{c\alpha n^2+d\beta n}\right|^2=
\begin{cases}
p^2 & \textrm{ if $p$ divides $c$ and $d$,}  \\
0 & \textrm{ if $p$ divides $c$ but not $d$, and }  \\
p & \textrm{ if $p$ does not divide $c$. }
\end{cases}
\]

Thus, to check formal duality using criterion~\eqref{condition1d} from
Definition~\ref{def:main}, we just need to verify that the system of
equations
\[
(c,d)=(j-k, j^2-k^2)
\]
has $p$ solutions if $c=d=0$, no solution if $c=0, d\ne 0$ and exactly
one solution if $c, d\ne 0$. The first two of these statements are
obvious. For the last one, note that we may solve $j+k = d/c$, which
leads to a unique solution $(j,k) = \left(\tfrac 12(\frac{d}{c}+c),
\tfrac 12(\frac{d}{c}-c)\right)$, since $2$ is invertible modulo $p$.
\end{proof}

\begin{remark} \label{remark:uniqueness}
Because $\alpha$ and $\beta$ can vary, the formal dual of a subset is not
unique, even modulo translation and automorphisms.
\end{remark}

\section{Structure theory in the cyclic case}

\subsection{Basic structure theory}

We begin with a few observations on the structure of formally dual sets.

The first basic observation is that if $S \subseteq G$ and $T \subseteq
\Ghat$ are formally dual, and $x \in G$, $y \in \Ghat$, then $S + x$ and $T +
y$ are also formally dual (since formal duality only cares about differences
of elements).

Let $G$ be a finite abelian group, and $H$ a subgroup of $G$.
Viewing $\Ghat = \Hom(G, S^1)$ and $\Hhat = \Hom(H,S^1)$, we
have a natural restriction map $\phi \colon \Ghat \to \Hhat$,
with kernel the annihilator of $H$, i.e.,
\[
H^\perp := \{y \in \Ghat \st \langle x, y \rangle = 1 \textrm{ for all } x \in H \}.
\]
Now, if $S \subseteq H$ and $T \subseteq \Hhat$ are formally dual subsets, we
may regard $S$ as a subset of $G$ and lift $T$ to $\Ghat$ using $\phi^{-1}$.

\begin{lemma} \label{lemma:liftingequiv}
The subsets $S \subseteq H$ and $T \subseteq \Hhat$ are formally dual if and
only if $S \subseteq G$ and $\phi^{-1}(T) \subseteq \Ghat$ are formally dual.
\end{lemma}

\begin{proof}
The easiest way to see this is to use condition \eqref{condition1d} of
Definition~\ref{def:main}, with the roles of $G$ and $\Ghat$ reversed.  It
says that $S \subseteq H$ and $T \subseteq \Hhat$ are formally dual iff for
all $x \in H$,
\[
\left| \frac{1}{|T|} \sum_{w \in T} \langle w, x \rangle \right|^2 =
\frac{1}{|S|} \cdot \# \{(v,v') \in S \times S \st x = v-v' \}.
\]

Under the above transformation from $(H, \Hhat, S, T)$ to $(G, \Ghat, S,
\phi^{-1}(T))$, the right side remains unchanged if $x \in H$, while the left
side becomes
\[
\left| \frac{1}{(G:H) |T|} \sum_{z \in \phi^{-1}(T)} \langle z, x \rangle \right|^2
= \left| \frac{1}{|T|} \sum_{w \in T} \langle w, x \rangle \right|^2
\]
since for every $z \in \Ghat$ mapping to $w \in \Hhat$ under $\phi$, we have
$\langle z, x \rangle = \langle w, x \rangle$, and there are exactly $(G:H)$
such $z$ for any $w$.  Thus, for $x \in H$ condition \eqref{condition1d}
holds for $(G, \Ghat, S, \phi^{-1}(T))$ iff it holds for $(H, \Hhat, S, T)$.

On the other hand, if $x \notin H$, then
\[
\# \{(v,v') \in S \times S \st x = v-v' \} = 0,
\]
since $S - S \subseteq H - H = H$.  The sum
\[
\sum_{z \in \phi^{-1}(T)} \langle z, x \rangle
\]
also vanishes: for each $t \in T$, let $t_0$ be any element of
$\phi^{-1}(t)$, and then
\[
\sum_{z \in \phi^{-1}(\{t\})} \langle z, x \rangle
= \sum_{y \in H^\perp} \langle y+t_0, x\rangle
= \langle t_0, x\rangle \sum_{y \in H^\perp} \langle y, x\rangle = 0,
\]
because $y \mapsto \langle y,x \rangle$ is a non-trivial character of
$H^\perp$, which sums to zero over $H^\perp$.  This completes the proof of
equivalence.
\end{proof}

In fact, this construction is reversible.

\begin{lemma} \label{prim}
Let $S \subseteq H \le G$ be formally dual to $T \subseteq \Ghat$. Then $T$
is invariant under addition by any element of $H^\perp$, and the image of $T$
under the restriction map $\phi$ is a formal dual to $S \subseteq H$ in $\Hhat$.
\end{lemma}

Here $H \le G$ means $H$ is a subgroup of $G$.

\begin{proof}
For $y \in \Hhat$, define its multiplicity by
\[
m(y)= \# \left( \phi^{-1} (y) \cap T \right).
\]
Evidently $0 \leq m(y) \leq (G:H)$ for all $y \in \Hhat$. We will
begin by refining this to $m(y)\in\{0, (G:H)\}$.  Recall that for each $x\in
G$,
\begin{equation} \label{eq:reverse}
\left|\sum_{w \in T} \langle x, w \rangle \right|^2 =
\frac{|T|^2}{|S|}\cdot \# \{(v, v') \in S \times S \st  x=v - v' \}.
\end{equation}
Summing this over all $x \in H$, the left side becomes
\[
\sum_{x\in H}\left|\sum_{y\in\Hhat}m(y) \langle x, y \rangle \right|^2
= \sum_{x\in H}\sum_{y, y'\in\Hhat} m(y) m(y') \langle x, y-y' \rangle.
\]
Interchanging the order of summation, we see that this equals
\[
|H|\sum_{y, y'\in\Hhat} m(y) m(y') \delta_{y, y'} = |H|\sum_{y\in\Hhat} m(y)^2.
\]
To simplify the right side of \eqref{eq:reverse} after summing over $x \in H$,
we observe that all differences of the form $v -
v'$ are automatically in $H$. We thus get
\[
|H|\sum_{y\in\Hhat}m(y)^2=|T|^2|S|.
\]
Using $|S||T|=|G|$ and $\sum_{y\in\Hhat}m(y)=|T|$ and canceling $|H|$, we may
rewrite this as
\[
\sum_{y\in\Hhat}m(y)^2=(G:H)\sum_{y\in\Hhat}m(y).
\]
Now for each individual $y\in\Hhat$ we have $m(y)^2 \leq
(G:H)m(y)$, with equality if and only if $m(y)\in\{0, (G:H)\}$. Hence
the previous equation is only possible if this is indeed the case for
all $y\in\Hhat$.

It follows that $T$ is invariant under translation by $H^\perp$, because for each $y$,
$\phi^{-1}(y)$ consists of an $H^\perp$-orbit of size $(G : H)$.  Thus, we
are in the situation covered by Lemma~\ref{lemma:liftingequiv}, and we
conclude that $S \subseteq H$ and $\phi(T) \subseteq \Hhat$ are formally
dual.
\end{proof}

The above results correspond to producing new formally dual configurations
in Euclidean space by taking a smaller underlying lattice. Let us say that
$S$ and $T$ are a \emph{primitive} pair of formally dual configurations if
$S$ is not contained in a coset of a proper subgroup of $G$ and $T$ is not
contained in a coset of a proper subgroup of $\Ghat$. In the classification
of formal duals, we may restrict to the primitive case.

\subsection{The $1$-dimensional case}

When $G$ is cyclic, we conjecture that there are no primitive formally dual
configurations except the trivial example and TITO.
We are able to prove the conjecture when $G = \Z/p^2\Z$, with
$p$ an odd prime.  The same is obviously true for $\Z/p\Z$, since the
product of the sizes of the dual configurations would be $p$.
By contrast, Theorem~\ref{theorem:gauss} shows that there
are nontrivial examples in $(\Z/p\Z)^2$.

\begin{proposition}
  Let $p$ be an odd prime. Then there are no primitive formally dual
  configurations in $G = \Z/p^2\Z$ and its dual.
\end{proposition}

\begin{proof}  If such configurations exist, then they must both
have size $p$. Let $S = \{v_1, \dots, v_p\}$ and $T = \{w_1, \dots,
w_p\}$ be formally dual, where we have identified $\widehat{G}$ with $\Z/p^2\Z$
via the pairing $\langle x, y \rangle = \zeta^{xy}$ with $\zeta = e^{2\pi
i/p^2}$.  We assume without loss of generality that $v_1=w_1=0$.

From the first condition of Definition~\ref{def:main}, we obtain
\[
\left| \sum_{i=1}^p \zeta^{yv_i} \right|^2 = p \cdot n_y,
\]
where we set $n_y = \# \{ (j,k) \st w_j - w_k = y\}$. That is,
\[
p + \sum_{i\neq j} \zeta^{y(v_i - v_j)} = p \cdot n_y.
\]
So $Z_y := \sum_{i \ne j} \zeta^{y(v_i - v_j)}$ is the rational integer
$p(n_y - 1)$. Now, note that as $y$ ranges over all the numbers modulo $p^2$
that are coprime to $p$, the algebraic numbers $Z_y$ are all conjugates of
each other. Since they are integers, they are all equal, and so are the
numbers $n_y$. Furthermore, we cannot have $n_y=0$ for all $y$ coprime to
$p$; otherwise all of $w_1,\dots,w_p$ would be multiples of $p$ (since $w_1 =
0$) and $T$ would be contained in a subgroup. Thus $n_y \geq 1$, and $Z_y
\geq 0$. But their sum
\[
\sum_{\gcd (y,p) =1} \sum_{i \neq j} \zeta^{y(v_i - v_j)} =
\sum_{i \neq j} \sum_{\gcd (y,p) =1} \zeta^{y(v_i - v_j)}
\]
equals zero, because the inner sum is zero for every pair $i \neq j$. (This
follows from $\sum_{j=1}^{p^2} \zeta^j = 0$ and $\sum_{j=1}^p \zeta^{pj} =
0$.) Therefore $n_y = 1$ for all $y$, which means the differences $w_i - w_j$
for $i \neq j$ cover all the $p(p-1)$ elements modulo $p^2$ that are coprime to $p$
exactly once. This is impossible by the following lemma, so we get a
contradiction.
\end{proof}

\begin{lemma} \label{lemma:psquareddiffsets}  Let $p$ be an odd prime.  Then
  there is no subset $S$ of $\Z/p^2 \Z$ whose difference set
  $\{x - y \st x, y \in S,\ x \neq y \}$ is the set of elements coprime
  to $p$.
\end{lemma}

\begin{proof}
Assume there is such a set $S$. Then the elements of $S$ must be distinct
modulo $p$, since otherwise some difference would be a multiple of $p$.
Without loss of generality $0 \in S$, since we can translate $S$ arbitrarily.
We list the elements as
\[ 
x_0 = 0, \quad x_1 = 1 + a_1 p, \quad  x_2 = 2 + a_2 p,\quad  \dots, \quad x_{p-1} = (p-1) + a_{p-1} p,
\]
where the integers $a_i$ are well defined modulo $p$. Now, among the
differences, the numbers congruent to $1$ modulo $p$ are
\begin{align*}
x_1 - x_0 &= 1 + a_1p, \\
x_2 - x_1 &=  1 + (a_2 - a_1) p, \\
&\vdots\\
x_{p-1} - x_{p-2} &= 1 + (a_{p-1} - a_{p-2})p, \\
x_0 - x_{p-1} &= p^2 - (p-1) - a_{p-1} p = 1 + (p-1 - a_{p-1}) p.
\end{align*}
Since these differences are all distinct modulo $p^2$, we need $a_1, a_2 -
a_1, \dots, p-1 - a_{p-1}$ to be distinct modulo $p$. Taking their
(telescoping) sum, we get
\[
p - 1 \equiv  0 + 1 + \dots + (p-1) = \frac{p(p-1)}{2} \pmod p,
\]
which is impossible for odd $p$.
\end{proof}

We thank Gregory Minton for providing the above short proof of the lemma.

\section{Non-existence of some formal duals} \label{section:nonexistence}

In this section, we show that some well-known packings do not have
formal duals.

\subsection{Barlow packings}

Recall that the Kepler conjecture was settled by Hales, \cite{H}
based partially on his work with Ferguson \cite{HF}. As a
result, the face-centered cubic lattice $A_3$ gives a densest
sphere packing in $\R^3$. It has uncountably many equally dense
competitors, the Barlow packings, obtained by layering the
densest planar arrangement (i.e., the hexagonal lattice $A_2$)
in different ways. The periodic packings among them are the
only periodic packings of maximal density in $\R^3$.  The
face-centered cubic lattice has a formal dual, namely its dual
lattice, and it is natural to ask whether the other periodic
Barlow packings have formal duals.
Proposition~\ref{prop:barlow} shows that they do not.

The periodic Barlow packings can be constructed as follows (see
\cite{CS1} for more details and a geometric description).  Let
$k$ be the number of hexagonal layers in a period. The $A_2$
lattice is spanned by two unit vectors $v_1$ and $v_2$, making
an angle of $\pi/3$, and the underlying lattice $\Lambda$ of
the Barlow packing is spanned by $v_1$, $v_2$, and $v_3$,
where $v_3$ is a vector of length $k\sqrt{2/3}$ that is
orthogonal to $v_1$ and $v_2$.  In addition to this lattice, we
need to specify how much each layer is offset. Let
$a_0,\dots,a_{k-1}$ be elements of $\{0,1,2\}$ with $a_j \ne
a_{j+1}$ for all $j$ (where we interpret indices modulo $k$).
Then the entire periodic configuration consists of the cosets
of $\Lambda$ by the translation vectors $a_j (v_1+v_2)/3 + j
v_3/k$.

\begin{proposition} \label{prop:barlow}
The only periodic Barlow packing that has a formal dual is the face-centered
cubic lattice.
\end{proposition}

The face-centered cubic case is when $k$ is a multiple of $3$, the sequence
$a_0,\dots,a_{k-1}$ is periodic modulo $3$, and $\{a_0,a_1,a_2\} =
\{0,1,2\}$.

For the proof of Proposition~\ref{prop:barlow}, consider any periodic Barlow
packing, with the notation established above. Transforming to the setting of
abelian groups and applying Lemma~\ref{prim}, we can let $G$ be the group
$\Z/3\Z \times \Z/k\Z$ (generated by $(v_1+v_2)/3$ and $v_3/k$ modulo
$\Lambda$).  The question becomes whether the subset $S = \{(a_j,j) \st 0 \le
j < k\}$ of $G$ has a formal dual.  We identify the dual group $\widehat{G}$
with $\Z/3\Z \times \Z/k\Z$ via the pairing
\[
\langle (a,b),(c,d) \rangle = \omega^{ac}\zeta^{bd},
\]
where $\omega = e^{2\pi i/3}$ and $\zeta = e^{2\pi i/k}$.

\begin{lemma}
If $S$ has a formal dual, then $k$ is a multiple of $3$.
\end{lemma}

\begin{proof}
Let $T$ be a formal dual of $S$, which must have $|T|=3$.  If we take
$y=(1,0)$ in
\[
\frac{|T|}{|S|^2} \left| \sum_{v \in S} \langle y,v \rangle
\right|^2 = \#\{(w,w') \in T : y = w-w'\},
\]
we find that
\[
\frac{3}{k^2} \left| \sum_{j=0}^{k-1} \omega^{a_j}
\right|^2
\]
is an integer, which must be $0$, $1$, $2$, or $3$.  It cannot be $3$ because
$a_0,\dots,a_{k-1}$ are not all equal.  If it is $1$ or $2$, then $\left|
\sum_{j=0}^{k-1} \omega^{a_j} \right|^2$ is $k^2/3$ or $2k^2/3$ and is also
an algebraic integer, so $k$ must be divisible by $3$.  Finally, if
$\sum_{j=0}^{k-1} \omega^{a_j}=0$, then the polynomial $\sum_{j=0}^{k-1}
x^{a_j}$ is divisible by $1+x+x^2$ in $\Z[x]$, and setting $x=1$ shows that
$k$ is a multiple of $3$.
\end{proof}

For the remainder of the proof of Proposition~\ref{prop:barlow}, suppose $S$
does have a formal dual $T$.  Because $k$ is a multiple of $3$, we can
replace $\omega$ with $\zeta^{k/3}$ and write the condition for formal
duality with $y=(r,s)$ as
\[
\frac{3}{k^2} \left| \sum_{j=0}^{k-1} \zeta^{ra_jk/3 + sj}
\right|^2 = \#\{(w,w') \in T \times T : (r,s) = w-w'\}.
\]

Without loss of generality, we can let $T = \{0, t_1, t_2\}$.  There are at
most six nonzero differences of elements of $T$, namely $\pm t_1$, $\pm t_2$,
and $\pm (t_1-t_2)$.  Thus, there can be at most six nonzero vectors $(r,s)$
for which
\[
\sum_{j=0}^{k-1} \zeta^{ra_jk/3 + sj} \ne 0.
\]

Our next step is to show that whenever $(r,s)$ satisfies
$\sum_j \zeta^{ra_jk/3 + sj} \ne 0$, its second coordinate $s$
must have a large factor in common with $k$.  For example, in
the face-centered cubic case $s$ is always divisible by $k/3$,
and this divisibility corresponds to the periodicity of
$a_0,\dots,a_{k-1}$ modulo $3$.

Let $m=k/3$, and write $\sum_j \zeta^{r a_j m + sj}$ in terms
of $\zeta' := \zeta^{\gcd(m,s)}$, which is a primitive root of
unity of order $k/\gcd(m,s)$.  The automorphisms of
$\Q(\zeta')$ are given by $\zeta' \mapsto \big(\zeta'\big)^u$
with $u$ a unit modulo $k/\gcd(m,s)$. These maps preserve
whether $\sum_j \zeta^{r a_j m + sj}$ vanishes, and they amount
to multiplying $y=(r,s)$ by $u$. Note that $us=u's$ in $\Z/k\Z$
iff $u \equiv u' \pmod{k/\gcd(k,s)}$, and $us \ne 0$ if $s \ne
0$.

\begin{lemma} \label{lemma:easydiv}
Given positive integers $a$ and $b$ with $a$ dividing $b$,
there exist $\varphi(a)$ units modulo $b$ that are distinct
modulo $a$.
\end{lemma}

Here $\varphi$ denotes the Euler totient function.
Lemma~\ref{lemma:easydiv} amounts to the standard fact that the
restriction map from $(\Z/b\Z)^\times$ to $(\Z/a\Z)^\times$ is
surjective; we will provide a proof for completeness.

\begin{proof}
Factor $b/a$ as $b'a'$, where $a'$ contains all the prime
factors that also divide $a$ and $b'$ contains all those that
do not.  Then units modulo $a$ are also units modulo $aa'$, and
we can use the Chinese remainder theorem to lift them to values
that are $1$ modulo $b'$ and the same modulo $aa'$.  The result
is $\varphi(a)$ units modulo $b$ that are distinct modulo $a$.
\end{proof}

Given a nonzero element $(r,s)$ for which $\sum_j \zeta^{r a_j
m + sj} \ne 0$, we can now apply Lemma~\ref{lemma:easydiv} with
$a = k/\gcd(k,s)$ and $b = k/\gcd(m,s)$ to find at least
$\varphi\big(k/\gcd(k,s)\big)$ distinct, nonzero elements
$(ur,us)$ of $\Z/3\Z \times \Z/k\Z$ such that $\sum_j \zeta^{ur
a_j m + usj} \ne 0$.  Thus, $\varphi\big(k/\gcd(k,s)\big) \le
6$, which implies $k/\gcd(k,s) \le 18$.

If $\varphi\big(k/\gcd(k,s)\big) \ge 3$ for some nonzero $(r,s) \in T-T$,
then at least three elements of $\{\pm t_1, \pm t_2, \pm (t_1-t_2)\}$ have
the same value of $\gcd(k,s)$ for their second coordinate $s$.  Call this
common value $g$.  It follows from the pigeonhole principle that $g$ divides
the second coordinate of at least two of $t_1$, $t_2$, and $t_1-t_2$, and
hence all three of them.  Therefore every element of $T$ has second
coordinate a multiple of $g$, and $k/g \le 18$.

The other possibility is that $\varphi\big(k/\gcd(k,s)\big) \le 2$ for all
nonzero $(r,s) \in T-T$.  Then $k/\gcd(k,s) \in \{1,2,3,4,6\}$ for all such
$(r,s)$, and the least common multiple of these numbers is $12$.  Letting $g$
be the greatest common divisor of $\gcd(k,s)$ for all nonzero $(r,s) \in
T-T$, we find that every element of $T$ has second coordinate a multiple of
$g$, with $k/g \le 12$.

Thus, in every case $T$ is contained in the subgroup of $\Z/3\Z
\times \Z/k\Z$ generated by $(1,0)$ and $(0,g)$, for some $g$
with $k/g \le 18$. By Lemma~\ref{prim}, $S$ must be invariant
under the annihilator of this subgroup, which is generated by
$(0,k/g)$. In other words, the layers in the Barlow packing are
periodic modulo $k/g$, where $k/g \le 18$.

This means we can assume without loss of generality that there
are at most $18$ layers (i.e., $k \le 18$).  Furthermore, we
can assume $a_0=0$ and $a_1=1$.  Then there are few enough
possibilities to enumerate them by computer, and one can check
that the integrality conditions
\[
\frac{3}{k^2} \left| \sum_{j=0}^{k-1} \zeta^{ra_jk/3 + sj}
\right|^2 \in \Z
\]
rule out all cases except the face-centered cubic lattice. This
completes the proof of Proposition~\ref{prop:barlow}.

\subsection{The Best packing in $\R^{10}$} \label{subsec:Best}

The Best packing is the densest known packing in $\R^{10}$. It
is a periodic configuration, consisting of $40$ translates of a
lattice. It can be constructed as the subset of $\Z^{10}$ that
reduces modulo $2$ to the nonlinear $(10,40,4)$ Best binary
code (see \cite[p.~140]{CS2}).

\begin{proposition} \label{prop:Best}
The Best configuration does not have a formal dual.
\end{proposition}

\begin{proof}
Again applying Lemma~\ref{prim}, we can assume $G =
(\Z/2\Z)^{10}$ and $S \subseteq G$ is the Best code. Since $|S|
= 40$ does not divide $|G| = 1024$, there cannot be a formal
dual.
\end{proof}

It remains an open question whether the Best packing has a radial formal dual
\cite[p.~185]{CS2}.  It seems unlikely that it has one, but radial formal
duality does not support the sort of structural
analysis we have used to prove Proposition~\ref{prop:Best}.

\section{Open questions}

We conclude with some open questions about formal duality.  Formal duality initially
arose in the simulations described in \cite{CKS}, and its occurrence there
remains unexplained: although our results in this paper substantially clarify
the algebraic foundations of this duality theory, they give no conceptual
explanation of why periodic energy minimization ground states in low
dimensions seem to exhibit formal duality. That is the most puzzling aspect
of the theory.

It would be interesting to classify all formally dual pairs.  Is every
example derived from the trivial construction, TITO, and the Gauss sum
construction by taking products and inflating the group (as in
Lemma~\ref{lemma:liftingequiv})?

TITO feels like a characteristic two relative of the Gauss sum construction,
but it occurs in $\Z/4\Z$ rather than $(\Z/2\Z)^2$.  Is there a unified
construction that subsumes TITO and the Gauss sum cases?

Conway and Sloane have given a conjectural list of all the ``tight'' packings
in up to nine dimensions \cite{CS1}.  Their list is believed to include all the
densest periodic packings in these dimensions.  Can one analyze which ones have formal duals,
perhaps by adapting the proof of Proposition~\ref{prop:barlow}?  Note that
the list contains at least a few non-lattice packings with formal duals,
namely $\Lambda_5^2$, $\Lambda_6^2$, and $\Lambda_7^3$, as shown in
\cite{CKS}.

\section*{Acknowledgments}

We thank Noam Elkies, Gregory Minton, and Peter Sarnak for helpful
conversations.

\bibliographystyle{amsalpha}

\end{document}